\documentclass[12pt]{article}

\usepackage{latexsym,amsmath,amscd,amssymb,graphics,float}
\usepackage{enumerate}

\usepackage{graphicx,lscape}

\usepackage[colorlinks]{hyperref}
\usepackage{url}

\begin{document}

\newenvironment{proof}[1][Proof]{\textbf{#1.} }{\ \rule{0.5em}{0.5em}}

\newtheorem{theorem}{Theorem}[section]
\newtheorem{definition}[theorem]{Definition}
\newtheorem{lemma}[theorem]{Lemma}
\newtheorem{remark}[theorem]{Remark}
\newtheorem{proposition}[theorem]{Proposition}
\newtheorem{corollary}[theorem]{Corollary}
\newtheorem{example}[theorem]{Example}

\numberwithin{equation}{section}
\newcommand{\ep}{\varepsilon}
\newcommand{\R}{{\mathbb  R}}
\newcommand\C{{\mathbb  C}}
\newcommand\Q{{\mathbb Q}}
\newcommand\Z{{\mathbb Z}}
\newcommand{\N}{{\mathbb N}}

\newcommand{\bfi}{\bfseries\itshape}

\newsavebox{\savepar}
\newenvironment{boxit}{\begin{lrbox}{\savepar}
\begin{minipage}[b]{15.5cm}}{\end{minipage}\end{lrbox}
\fbox{\usebox{\savepar}}}

\title{{\bf Controlling the stability of periodic orbits of completely integrable systems}}
\author{R\u{a}zvan M. Tudoran}

\date{}
\maketitle \makeatother

\begin{abstract}
We provide a constructive method designed in order to control the stability of a given periodic orbit of a general completely integrable system. The method consists of a specific type of perturbation, such that the resulting perturbed system becomes a codimension-one dissipative dynamical system which also admits that orbit as a periodic orbit, but whose stability can be a-priori prescribed. The main results are illustrated in the case of a three dimensional dissipative perturbation of the harmonic oscillator, and respectively Euler's  equations form the free rigid body dynamics.
\end{abstract}

\medskip

\textbf{AMS 2000}: 37C27; 37C75; 34C25; 37J35.


\section{Introduction}
\label{section:one}

The main purpose of this work is to provide a constructive method of controlling the stability of periodic orbits of completely integrable systems. The method consists of a specific type of perturbation, such that the resulting perturbed system becomes a codimension-one dissipative dynamical system. 

The controllability procedure is based on an explicit formula for the characteristic multipliers of a given periodic orbit of a general codimension-one dissipative dynamical system. Even if this explicit formula is not the main purpose of the paper, it can be considered as the core of the article, since all main results are actually based on it. Recall that the explicit knowledge of the characteristic multipliers of a periodic orbit, is extremely useful for the study of its stability (for details regarding the stability of periodic orbits see e.g., \cite{hartman}, \cite{verhulst}). 

Because of the local nature of the main results, one can suppose that we work on an open subset $U\subseteq \mathbb{R}^n$. More precisely, let $\dot x =X(x)$, $X\in\mathfrak{X}(U)$, be a given codimension-one dissipative dynamical system, i.e., there exists $k,p\in\mathbb{N}$ such that $k+p=n-1$, and some smooth functions $I_1,\dots,I_k, D_1,\dots, D_p$, $h_1$, $\dots$, $h_p \in\mathcal{C}^{\infty}(U,\mathbb{R})$ such that the vector field $X$ conserves $I_1,\dots,I_k$, and dissipates $D_1,\dots, D_p$ with associated dissipation rates $h_1 D_1, \dots, h_p D_p$. Suppose that $\Gamma:=\{\gamma(t)\subset U : 0\leq t\leq T \}$ is a $T-$periodic orbit of $\dot x =X(x)$ such that $\Gamma\subset ID^{-1}(\{0\})$, and moreover, $0 \in \mathbb{R}^{n-1}$ is a regular value of the map $ID:=(I_1,\dots,I_k,D_1,\dots,D_p):U\subseteq \mathbb{R}^{n}\rightarrow \mathbb{R}^{n-1}$. 

In the above hypothesis, the first result states that, if $$\nabla I_1(\gamma(t)),\dots,\nabla I_k(\gamma(t)),\nabla D_1(\gamma(t)),\dots, \nabla D_p (\gamma(t)), X(\gamma(t))$$ are linearly independent for each $0\leq t \leq T$, then the characteristic multipliers of the periodic orbit $\Gamma$ are, $\underbrace{1,\dots,1}_{k+1 \ \mathrm{times}}, \exp\left(\int_{0}^{T}h_1 (\gamma(s))ds\right),\dots, \exp\left(\int_{0}^{T}h_p (\gamma(s))ds\right)$. 

The second result of this work is a consequence of the explicit computation of the characteristic multipliers of a given periodic orbit, and consists of two stability results. More precisely, if there exists $i_0 \in\{1,\dots, p\}$ such that $\int_{0}^{T}h_{i_0} (\gamma(s))ds>0$, then the periodic orbit $\Gamma$ is unstable. 

On the other hand (supposing that $0\in \mathbb{R}^{k}$ is a regular value of the map $I:=(I_1,\dots,I_k)$), if $\int_{0}^{T}h_1 (\gamma(s))ds<0$, $\dots$, $\int_{0}^{T}h_p (\gamma(s))ds<0$, then the periodic orbit $\Gamma$ is orbitally phase asymptotically stable, with respect to perturbations along the invariant manifold $I^{-1}(\{0\})$.

The third result of this article, which is also the main result, provides a method to partially stabilize a given periodic orbit of a completely integrable dynamical system. More precisely, to a given completely integrable system, and respectively a given periodic orbit, we explicitly associate a dissipative dynamical system admitting the same periodic orbit, and moreover, this periodic orbit is orbitally phase asymptotically stable, relative to a certain dynamically invariant set. Note that dissipative perturbations were also used in stabilization procedures of equilibrium states of Hamiltonian systems, see e.g., \cite{rtud}, \cite{bc}.

The structure of the paper is the following. In the second section, one recalls a characterization of codimension-one dissipative dynamical systems, that will be used in the next sections. The third section is dedicated to the explicit computation of the characteristic multipliers of a given periodic orbit of a general codimension-one dissipative dynamical system. The fourth section uses the results from the previous section, in order to provide sufficient conditions to guarantee the partial stability, and respectively the instability of periodic orbits of codimension-one dissipative dynamical systems. The last section contains an explicit method to stabilize (relatively to a certain dynamically invariant set) a given periodic orbit of a completely integrable dynamical system.

\section{Codimension-one dissipative dynamical systems}

In this short section we recall some results concerning the codimension-one dissipative dynamical systems. For more details regarding the characterization of general dissipative dynamical systems see e.g., \cite{rtudoran}. 

Recall that by a codimension-one dissipative dynamical system (defined eventually on an open subset $U\subseteq \mathbb{R}^{n}$), we mean a dynamical system $\dot x =X(x)$, $X\in\mathfrak{X}(U)$, for which there exist $k,p\in\mathbb{N}$, with $k+p=n-1$, and some smooth functions $I_1,\dots,I_k, D_1,\dots, D_p$, $h_1$, $\dots$, $h_p \in\mathcal{C}^{\infty}(U,\mathbb{R})$, such that the vector field $X$ conserves $I_1,\dots,I_k$, and dissipates $D_1,\dots, D_p$, with associated dissipation rates $h_1 D_1, \dots, h_p D_p$, i.e., $\mathcal{L}_{X}I_1 =\dots= \mathcal{L}_{X}I_k =0$, and respectively $\mathcal{L}_{X}D_1 =h_1 D_1$, $\dots$, $\mathcal{L}_{X}D_p =h_p D_p$, where the notation $\mathcal{L}_{X}$ stands for the Lie derivative along the vector field $X$. Note that $\mathcal{L}_{X} f =\langle X,\nabla f \rangle$, where $f\in\mathcal{C}^{\infty}(U,\mathbb{R})$ is an arbitrary smooth function, and $\nabla$ states for the gradient operator associated to the standard inner product on $\mathbb{R}^n$, namely $\langle\cdot,\cdot\rangle$.

Let us recall now a result from \cite{rtudoran} regarding the local structure of a codimension-one dissipative dynamical system.

\begin{theorem}[\cite{rtudoran}]\label{DSC}
Let $k,p\in \mathbb{N}$ be two natural numbers such that $k+p = n-1$, and let $I_1,\dots, I_k, D_1,\dots, D_p, h_1, \dots, h_p \in \mathcal{C}^{\infty}(U,\mathbb{R})$ be a given set of smooth functions defined on an open subset $U\subseteq \mathbb{R}^{n}$, such that $\{\nabla I_1 ,\dots,\nabla I_k , \nabla D_1 , \dots,\nabla D_p\}\subset \mathfrak{X}(U)$ forms a set of pointwise linearly independent vector fields on $U$.

Then the smooth vector fields $X\in \mathfrak{X}(U)$ which verify simultaneously the conditions
\begin{equation}\label{edr}
\left\{\begin{array}{l}
\mathcal{L}_{X}I_1= \dots = \mathcal{L}_{X}I_k = 0,\\
\mathcal{L}_{X}D_1 = h_1 D_1, \dots, \mathcal{L}_{X}D_p = h_p D_p,\\
\end{array}\right.
\end{equation}
are characterized as follows
$$
X = X_0 + \nu \left[\star\left(\bigwedge_{j=1}^{p} \nabla D_j\wedge\bigwedge_{l=1}^{k} \nabla I_l \right)\right],
$$
where $\nu\in\mathcal{C}^{\infty}(U,\mathbb{R})$ is an arbitrary rescaling function,
$$
X_{0}=\left\| \bigwedge_{i=1}^{p} \nabla D_i\wedge\bigwedge_{j=1}^{k} \nabla I_j \right\|_{n-1}^{-2}\cdot\sum_{i=1}^{p}(-1)^{n-i}h_i D_i \Theta_i, 
$$
$$
\Theta_i = \star\left[ \bigwedge_{j=1, j\neq i}^{p} \nabla D_j \wedge \bigwedge_{l=1}^{k} \nabla I_l  \wedge\star\left(\bigwedge_{j=1}^{p} \nabla D_j\wedge\bigwedge_{l=1}^{k} \nabla I_l \right)\right],
$$
and "$\star$" stands for the Hodge star operator for multivector fields.
\end{theorem}

\bigskip
\begin{remark}[\cite{rtudoran}]\label{ret}
The vector field $X_0$ is itself a solution of the system \eqref{edr}, while the vector field $\star\left(\bigwedge_{j=1}^{p} \nabla D_j\wedge\bigwedge_{l=1}^{k} \nabla I_l \right)$ is a solution of the homogeneous system
\begin{equation*}
\left\{\begin{array}{l}
\mathcal{L}_{X}I_1= \dots = \mathcal{L}_{X}I_k = 0,\\
\mathcal{L}_{X}D_1 = \dots =\mathcal{L}_{X}D_p = 0.\\
\end{array}\right.
\end{equation*}

\end{remark}

\bigskip
Let us now recall from \cite{rtudoran} a consequence of the above theorem, which gives a local characterization of completely integrable systems.

\begin{remark}[\cite{rtudoran}]\label{CIS}
In the case when $p=0$ (and consequently $k=n-1$), the dynamical system generated by the vector field $X$ will be completely integrable, and the conclusion of Theorem \eqref{DSC} becomes:

The smooth vector fields $X\in \mathfrak{X}(U)$ which verify the conditions
\begin{equation*}
\mathcal{L}_{X}I_1= \dots = \mathcal{L}_{X}I_{n-1} = 0,\\
\end{equation*}
are given by
$$
X = \nu \left[\star\left(\nabla I_1 \wedge\dots\wedge\nabla I_{n-1} \right)\right],
$$
where $\nu\in\mathcal{C}^{\infty}(U,\mathbb{R})$ is a smooth arbitrary function.
\end{remark}

\section{The characteristic multipliers of periodic orbits of codimension-one dissipative dynamical systems}

In this section we compute explicitly the characteristic multipliers of periodic orbits of codimension-one dissipative dynamical systems. Let us recall first that for a general dynamical system $\dot x=X(x)$, generated by a smooth vector field $X\in \mathfrak{X}(U)$, defined on an open subset $U\subseteq \mathbb{R}^{n}$, and respectively a given $T-$periodic orbit $\Gamma:=\{\gamma(t)\subset U : 0\leq t\leq T \}$, the characteristic multipliers of $\Gamma$ are the eigenvalues of the fundamental matrix $u(T)$, where $u$ is the solution of the variational equation $$\dfrac{du}{dt}=DX(\gamma(t))u(t), \ u(0)=I_{n,n},$$ and $I_{n,n}$ stands for the identity matrix of dimensions $n\times n$. Recall that since $\Gamma$ is a periodic orbit, $1$ will be always a characteristic multiplier of $\Gamma$, (see e.g., \cite{moser}). Taking into account the complexity of the variational equation, the computation of characteristic multipliers in general is almost impossible, since there exist no general methods to solve explicitly the variational equation. 

One of the main results of this paper is to complete this task for the class of codimension-one dissipative dynamical systems, if one knows an explicit parameterization of the periodic orbit to be analyzed. In order to do that we will use the following result from \cite{gasul}.

\begin{theorem}[\cite{gasul}]\label{TR}
Let $\Gamma=\{\gamma(t)\subset U : 0\leq t\leq T \}$ be a $T-$periodic orbit of a dynamical system $\dot x=X(x)$. Consider a smooth function $f:U\subseteq \mathbb{R}^{n}\rightarrow \mathbb{R}^{n-1}$, $f=(f_1,\dots,f_{n-1})^{\top}$, such that:
\begin{itemize}
\item $\Gamma$ is contained in $\bigcap_{i=1}^{n-1}\{f_i (x)=0\}$,
\item the crossing of all the manifolds $\{f_i (x)=0\}$ for $i\in\{1,\dots,n-1\}$ are transversal over $\Gamma$,
\item there exists a $(n-1)\times (n-1)$ matrix $k(x)$ of real functions satisfying:
\begin{equation}\label{ED}
Df(x)X(x)=k(x)f(x).
\end{equation}
\end{itemize}
Let $v(t)$ be the $(n-1)\times (n-1)$ fundamental matrix solution of 
\begin{equation}\label{EFD}
\dfrac{dv}{dt}=k(\gamma(t))v(t), \ v(0)=I_{n-1,n-1},
\end{equation}
where $I_{n-1,n-1}$ stands for the identity matrix of dimensions $(n-1)\times (n-1)$.
Then the characteristic multipliers of $\Gamma$ are $\{1\}\cup \sigma(v(T))$, where $\sigma(v(T))$ stands for the spectrum of $v(T)$.
\end{theorem}

Using the above result, we will compute explicitly the characteristic multipliers of a given periodic orbit of a codimension-one dissipative dynamical system. Let us state now the main result of this section.

\begin{theorem}\label{MT}
Let $\dot x= X(x)$ be a codimension-one dissipative dynamical system generated by a smooth vector field $X\in\mathfrak{X}(U)$ defined eventually on an open subset $U\subseteq \mathbb{R}^{n}$, such that there exist $k,p\in\mathbb{N}$, $k+p=n-1$, and respectively $I_1,\dots,I_k, D_1,\dots, D_p$, $h_1$, $\dots$, $h_p \in\mathcal{C}^{\infty}(U,\mathbb{R})$ such that $\mathcal{L}_{X}I_1 =\dots= \mathcal{L}_{X}I_k =0$, and $\mathcal{L}_{X}D_1 =h_1 D_1$, $\dots$, $\mathcal{L}_{X}D_p =h_p D_p$. Suppose that $\Gamma=\{\gamma(t)\subset U : 0\leq t\leq T \}$ is a $T-$periodic orbit of $\dot x= X(x)$, such that the following conditions hold true:
\begin{itemize}
\item $\Gamma\subset ID^{-1}(\{0\})$, and $0 \in \mathbb{R}^{n-1}$ is a regular value of the map $$ID=(I_1,\dots,I_k,D_1,\dots,D_p):U\subseteq \mathbb{R}^{n}\rightarrow \mathbb{R}^{n-1},$$
\item $\nabla I_1(\gamma(t)),\dots,\nabla I_k(\gamma(t)),\nabla D_1(\gamma(t)),\dots, \nabla D_p (\gamma(t)), X(\gamma(t))$ are linearly \\
independent for each $0\leq t \leq T$.
\end{itemize}
Then, the characteristic multipliers of the periodic orbit $\Gamma$ are $$\underbrace{1,\dots,1}_{k+1 \ \mathrm{times}}, \exp\left(\int_{0}^{T}h_1 (\gamma(s))ds\right), \dots, \exp\left(\int_{0}^{T}h_p (\gamma(s))ds\right).$$
\end{theorem}
\begin{proof}
The first two conditions from the hypothesis of Theorem \eqref{TR} obviously imply the first two conditions of this theorem. Let us construct now a matrix $(n-1)\times (n-1)$, $k(x)$, which verifies the equivalent of equation \eqref{ED}, namely the equation
\begin{equation}\label{EDD}
D(ID)^{\top}(x)X(x)=k(x)(ID)^{\top}(x),
\end{equation}
where by $(ID)^{\top}$ we mean the transpose of the matrix $ID=(I_1,\dots,I_k,D_1,\dots,D_p)$.

If we define
$
k=\left[ {\begin{array}{*{20}c}
   O_{k,k} & O_{k,p}  \\
   O_{p,k} & H_{p,p}  \\
\end{array}} \right],
$
where $O_{k,k},O_{k,p},O_{p,k}$ stands for the null matrices of dimensions $k \times k, k \times p, p \times k,$ and $H_{p,p}=\operatorname{diag}[h_1,\dots,h_p ]$, then the equation \eqref{EDD} can be equivalently written as follows
\begin{equation}\label{split}
\left\{\begin{array}{l}
\langle\nabla I_1,X\rangle=\dots=\langle\nabla I_k,X\rangle=0,\\
\langle\nabla D_1,X\rangle=h_1 D_1,\dots, \langle\nabla D_p,X\rangle=h_p D_p.\\
\end{array}\right.
\end{equation}

Hence, the matrix $k$ is a solution of equation \eqref{EDD}, since the system \eqref{split} is obviously equivalent to 
\begin{equation*}
\left\{\begin{array}{l}
\mathcal{L}_{X}I_1= \dots = \mathcal{L}_{X}I_k = 0,\\
\mathcal{L}_{X}D_1 = h_1 D_1, \dots, \mathcal{L}_{X}D_p = h_p D_p,\\
\end{array}\right.
\end{equation*}
which is by hypothesis verified by the vector field $X$.

Let us solve now the equation \eqref{EFD} associated to the matrix $k$. In order to solve the equation, let us split first the $(n-1)\times (n-1)$ matrix $v$ as follows
$
v=\left[ {\begin{array}{*{20}c}
   v_{k,k} & v_{k,p}  \\
   v_{p,k} & v_{p,p}  \\
\end{array}} \right],
$
where the blocks $v_{k,k},v_{k,p},v_{p,k},v_{p,p}$ have dimension $k\times k$, $k\times p$, $p\times k$, and respectively $p\times p$. Using the same type of splitting, the initial condition matrix $v(0)=I_{n-1,n-1}$ splits as follows
$$
v(0)=\left[ {\begin{array}{*{20}c}
   v_{k,k}(0) & v_{k,p}(0)  \\
   v_{p,k}(0) & v_{p,p}(0)  \\
\end{array}} \right]
=\left[ {\begin{array}{*{20}c}
   I_{k,k }& O_{k,p}  \\
   O_{p,k} & I_{p,p}  \\
\end{array}} \right],
$$
where $I_{k,k}, I_{p,p}$ stands for the identity matrices of dimensions $k\times k$, and respectively $p\times p$.
Consequently, the equation \eqref{EFD} becomes

\begin{equation*}
\left\{\begin{array}{l}
\left[ {\begin{array}{*{20}c}
   \dot{v}_{k,k}(t) & \dot{v}_{k,p}(t)  \\
   \dot{v}_{p,k}(t) & \dot{v}_{p,p}(t)  \\
\end{array}} \right]
=\left[ {\begin{array}{*{20}c}
   O_{k,k} & O_{k,p}  \\
   O_{p,k} & H_{p,p}(\gamma(t)) \\
\end{array}} \right]\cdot\left[ {\begin{array}{*{20}c}
   v_{k,k}(t) & v_{k,p}(t)  \\
   v_{p,k}(t) & v_{p,p}(t)  \\
\end{array}} \right], \\
\left[ {\begin{array}{*{20}c}
   v_{k,k}(0) & v_{k,p}(0)  \\
   v_{p,k}(0) & v_{p,p}(0)  \\
\end{array}} \right]=\left[ {\begin{array}{*{20}c}
   I_{k,k} & O_{k,p}  \\
   O_{p,k} & I_{p,p}  \\
\end{array}} \right],
\end{array}\right.
\end{equation*}
or equivalently
\begin{equation*}
\left\{\begin{array}{l}
\dot{v}_{k,k}(t)=O_{k,k},\\
\dot{v}_{k,p}(t)=O_{k,p},\\
\dot{v}_{p,k}(t)=H_{p,p}(\gamma(t)) v_{p,k}(t),\\
\dot{v}_{p,p}(t)=H_{p,p}(\gamma(t)) v_{p,p}(t),\\
v_{k,k}(0)=I_{k,k},\\
v_{k,p}(0)=O_{k,p},\\
v_{p,k}(0)=O_{p,k},\\
v_{p,p}(0)=I_{p,p},
\end{array}\right.
\end{equation*}
where $H_{p,p}(\gamma(t))=\operatorname{diag}\left[h_1(\gamma(t)),\dots,h_p(\gamma(t))\right]$, and $t\in[0,T]$.

Using standard ODE techniques, one obtains the unique solution $$v:[0,T]\rightarrow GL(n-1,\mathbb{\R}),$$ 
$$
v(t)=\left[ {\begin{array}{*{20}c}
   I_{k,k} & O_{k,p}  \\
   O_{p,k} & v_{p,p}(t)  \\
\end{array}} \right],
$$
where $$v_{p,p}(t)=\operatorname{diag}\left[\exp\left(\int_{0}^{t}h_1(\gamma(s))ds\right),\dots,\exp\left(\int_{0}^{t}h_p(\gamma(s))ds\right)\right].$$

Since from Theorem \eqref{TR}, the characteristic multipliers of the periodic orbit $\Gamma$ are given by $\{1\}\cup \sigma(v(T))$, we obtain the conclusion. 
\end{proof}

\bigskip
Note that if $p=0$, we recover a classical result concerning the characteristic multipliers of periodic orbits of completely integrable systems. More precisely, if $p=0$, then the dynamical system $\dot x=X(x)$ from Theorem \eqref{MT} becomes completely integrable, and using the conclusion of Theorem \eqref{MT} we get that the characteristic multipliers of any periodic orbit of a completely integrable system, are all equal to one.

\section{Some stability results regarding the periodic orbits of codimension-one dissipative dynamical systems}

This section has two main purposes, namely, the first purpose is to provide sufficient conditions to guarantee the partial orbital asymptotic stability with asymptotic phase of periodic orbits of a codimension-one dissipative dynamical system, whereas the second purpose is to give sufficient conditions to guarantee the instability of periodic orbits of a codimension-one dissipative dynamical system.

Let us start by recalling some definitions and also some general results concerning the stability of the periodic orbits of a general dynamical system. In order to do that, let $\dot x =X(x)$ be a dynamical system generated by a smooth vector field $X\in\mathfrak{X}(U)$, defined eventually on an open subset $U\subseteq \mathbb{R}^{n}$. Suppose $\Gamma =\{\gamma(t)\subset U : 0\leq t\leq T \}$ is a $T-$periodic orbit of $\dot x =X(x)$.

\begin{definition}
\begin{itemize}
\item The periodic orbit $\Gamma$ is called \textbf{orbitally stable} if, given $\varepsilon >0$ there exists a $\delta >0$ such that $\operatorname{dist}(x(t,x_0),\Gamma)<\varepsilon $ for all $t>0$ and for all $x_0 \in U$ such that $\operatorname{dist}(x_0,\Gamma)<\delta $.
\item The periodic orbit $\Gamma$ is called \textbf{unstable} if it is not orbitally stable.
\item The periodic orbit $\Gamma$ is called \textbf{orbitally asymptotically stable} if it is orbitally stable and (by choosing $\delta$ smaller if necessary), $\operatorname{dist}(x(t,x_0),\Gamma)\rightarrow 0$ as $t\rightarrow \infty$.
\item The periodic orbit $\Gamma$ is called \textbf{orbitally phase asymptotically stable}, if it is asymptotically orbitally stable and there is a $\delta >0$ such that for each $x_0 \in U$ with $\operatorname{dist}(x_0,\Gamma)<\delta $, there exists $\theta_{0}=\theta_0 (x_0)$ such that $$\lim_{t\rightarrow \infty}\|x(t,x_0)-\gamma(t+\theta_{0})\|=0.$$
\end{itemize}
\end{definition}

Let us now recall a classical result which gives some sufficient conditions to guarantee the stability/instability of a periodic orbit in terms of its characteristic multipliers. For more details regarding these results see e.g., \cite{hartman}, \cite{verhulst}.

\begin{theorem}\label{AW}
Suppose $\Gamma =\{\gamma(t)\subset U : 0\leq t\leq T \}$ is a $T-$periodic orbit of the dynamical system $\dot x =X(x)$ generated by a smooth vector field $X\in\mathfrak{X}(U)$, defined eventually on an open subset $U\subseteq \mathbb{R}^{n}$. 
\begin{itemize}
\item If the characteristic multiplier one is simple (has multiplicity one) and the rest of the characteristic multipliers of the periodic orbit $\Gamma$ have all of them the modulus strictly less then one, then the periodic orbit $\Gamma$ is asymptotically orbitally stable with asymptotic phase.
\item If there exists at least one characteristic multiplier of the periodic orbit $\Gamma$, whose modulus is strictly greater then one, then the periodic orbit $\Gamma$ is unstable. 
\end{itemize}
\end{theorem}

Let us now state the main result of this section, which is a generalization of the above result in the case when the characteristic multiplier one is not simple.
\begin{theorem}\label{SPO}
Let $\dot x= X(x)$ be a codimension-one dissipative dynamical system generated by a smooth vector field $X\in\mathfrak{X}(U)$ defined eventually on an open subset $U\subseteq \mathbb{R}^{n}$, such that there exists $k,p\in\mathbb{N}$ with $p>0$, $k+p=n-1$, and respectively $I_1,\dots,I_k, D_1,\dots, D_p$, $h_1$, $\dots$, $h_p \in\mathcal{C}^{\infty}(U,\mathbb{R})$ such that $\mathcal{L}_{X}I_1 =\dots= \mathcal{L}_{X}I_k =0$, and $\mathcal{L}_{X}D_1 =h_1 D_1$, $\dots$, $\mathcal{L}_{X}D_p =h_p D_p$. Suppose $\Gamma=\{\gamma(t)\subset U : 0\leq t\leq T \}$ is a $T-$periodic orbit of $\dot x= X(x)$, such that the following conditions hold true:
\begin{itemize}
\item $\Gamma\subset ID^{-1}(\{0\})$, and $0 \in \mathbb{R}^{n-1}$ is a regular value of the map $$ID=(I_1,\dots,I_k,D_1,\dots,D_p):U\subseteq \mathbb{R}^{n}\rightarrow \mathbb{R}^{n-1},$$
\item $\nabla I_1(\gamma(t)),\dots,\nabla I_k(\gamma(t)),\nabla D_1(\gamma(t)),\dots, \nabla D_p (\gamma(t)), X(\gamma(t))$ are linearly\\
 independent for each $0\leq t \leq T$.
\end{itemize}

Then, if moreover $0 \in \mathbb{R}^{k}$ is a regular value of the map $I=(I_1,\dots,I_k):U\subseteq \mathbb{R}^{n}\rightarrow \mathbb{R}^{k},$ and if $$\int_{0}^{T}h_1 (\gamma(s))ds<0, \dots, \int_{0}^{T}h_p (\gamma(s))ds<0,$$ then the periodic orbit $\Gamma$ is orbitally phase asymptotically stable, with respect to perturbations along the invariant manifold $I^{-1}(\{0\})$. 

On the other hand, if there exists $i_0 \in \{1,\dots,p\}$ such that $\int_{0}^{T}h_{i_0} (\gamma(s))ds>0$, then the periodic orbit $\Gamma$ is unstable.
\end{theorem}
\begin{proof}
Recall from Theorem \eqref{MT} that the characteristic multipliers of the periodic orbit $\Gamma$ of the vector field $X$ are $$\underbrace{1,\dots,1}_{k+1 \ \mathrm{times}}, \exp\left(\int_{0}^{T}h_1 (\gamma(s))ds\right), \dots, \exp\left(\int_{0}^{T}h_p (\gamma(s))ds\right).$$

By a classical result concerning the properties of characteristic multipliers in the presence of first integrals (see e.g., \cite{moser}) we have that, if the common level set of the first integrals $I_1,\dots,I_k$, containing $\Gamma$, is a smooth manifold, then the characteristic multipliers of $\Gamma$ as a periodic orbit of the restriction of the vector field $X$ to this dynamically invariant manifold are the following: $1$ (due to the fact that $\Gamma$ is a periodic orbit also for the restriction of $X$), and respectively the rest of $n-k-1$ characteristic multipliers of $\Gamma$ as a periodic orbit of $X$. Recall that $\Gamma$ as a periodic orbit of $X$ has $k+1$ characteristic multipliers equal to one ($k$ of them associated to the first integrals $I_1,\dots,I_k$, and one due to the fact that $\Gamma$ is a periodic orbit), and respectively some other $n-k-1$ characteristic multipliers (possible some of them also being equal to one). 

Consequently, if $0\in \mathbb{R}^{k}$ is a regular value of the map $I:=(I_1,\dots,I_k)$, then the dynamical system $\dot x = X|_{I^{-1}(\{0\})}(x),$ given by the restriction of the vector field $X$ to the dynamically invariant manifold $I^{-1}(\{0\})$, admits $\Gamma$ as periodic orbit (by dynamical invariance), and the associated characteristic multipliers are $$1, \exp\left(\int_{0}^{T}h_1 (\gamma(s))ds\right),\dots,\exp\left(\int_{0}^{T}h_p (\gamma(s))ds\right).$$

Hence, one can apply the Theorem \eqref{AW} for the dynamical system generated by the vector field $X|_{I^{-1}(\{0\})}$ and respectively for the periodic orbit $\Gamma$, and conclude the corresponding stability/instability results. Consequently, by dynamical invariance, the same conclusions hold true also for the periodic orbit $\Gamma$ associated to the original vector field $X$ with respect to perturbations along the invariant manifold $I^{-1}(\{0\})$. 

More precisely, if $\int_{0}^{T}h_1 (\gamma(s))ds<0$, $\dots$, $\int_{0}^{T}h_p (\gamma(s))ds<0$, then the characteristic multipliers of the periodic orbit $\Gamma$ of the vector field $X|_{I^{-1}(\{0\})}$ have the following properties: the characteristic multiplier one is simple (its multiplicity is one), and the rest of characteristic multipliers have modulus strictly less then one, and hence the periodic orbit is orbitally phase asymptotically stable. Hence, because of dynamical invariance, the same conclusion holds in the case of the vector field $X$ with respect to perturbations along the invariant manifold $I^{-1}(\{0\})$.

On the other hand (even if $0\in \mathbb{R}^{k}$ it is not a regular value of the map $I:=(I_1,\dots,I_k)$), if there exists $i_0 \in\{1,\dots, p\}$ such that $\int_{0}^{T}h_{i_0} (\gamma(s))ds>0$, we obtain directly from Theorem \eqref{AW} that the periodic orbit $\Gamma$ of the vector field $X$, it is unstable.
\end{proof}

\bigskip
Let us illustrate now the results of the above theorem in the case of a three dimensional dissipative perturbation of the harmonic oscillator. 

\begin{example}\label{DHO}
Let $\dot{\mathbf{x}}=X(\mathbf{x})$, $\mathbf{x}=(x,y,z)\in\mathbb{R}^{3}$, be the dynamical system generated by the smooth vector field  $$X(x,y,z)=y\partial_{x}-x\partial_{y}+zh(x,y,z)\partial_{z},$$ 
where $h\in\mathcal{C}^{\infty}(\mathbb{R}^{3},\mathbb{R})$ is an arbitrary given smooth real function. 

Then, the set $\Gamma=\{\gamma(t)=(\sin t,\cos t,0): 0\leq t\leq 2\pi\}$ is a $2\pi-$periodic orbit of the dynamical system $\dot{\mathbf{x}}=X(\mathbf{x})$. 

Note that the above defined dynamical system is a codimension-one dissipative system, associated with the following data
\begin{itemize}
\item $I:\mathbb{R}^{3}\rightarrow \mathbb{R}$, $I(x,y,z)=x^2 +y^2 -1$,
\item $D:\mathbb{R}^{3}\rightarrow \mathbb{R}$, $D(x,y,z)=z$,
\item $h:\mathbb{R}^{3}\rightarrow \mathbb{R}$, 
\end{itemize}
since, $\mathcal{L}_{X}I=0$ and $\mathcal{L}_{X}D=hD$.

The hypothesis of the Theorem \eqref{SPO} are verified since
\begin{itemize}
\item $\Gamma \subset ID^{-1}(\{(0,0)\})$,
\item $(0,0)$ is a regular value of the map $ID:\mathbb{R}^{3}\rightarrow \mathbb{R}^{2}$, $ID(x,y,z)=(x^2 +y^2 -1,z)$,
\item $\nabla I(\sin t,\cos t,0),\nabla D(\sin t,\cos t,0),X(\sin t,\cos t,0)$ are linearly independent \\
for each $t\in[0,2\pi]$,
\item $0$ is a regular value of the map $I$.
\end{itemize}

Hence, by Theorem \eqref{SPO} we obtain the following conclusions:
\begin{itemize}
\item if $\int_{0}^{2\pi}h(\sin t,\cos t,0)dt<0$, then the periodic orbit $\Gamma$ is orbitally phase asymptotically stable, with respect to perturbations along the cylinder $I^{-1}(\{0\})$,
\item if $\int_{0}^{2\pi}h(\sin t,\cos t,0)dt>0$, then the periodic orbit $\Gamma$ is unstable.
\end{itemize}
\end{example}

\section{Orbitally asymptotically stabilizing the periodic orbits of completely integrable dynamical systems}

The purpose of this section is to apply the results from the previous section in order to partially orbitally asymptotically stabilize, a given periodic orbit of a completely integrable dynamical system. In order to do that, let us consider a completely integrable dynamical system $\dot x=X(x)$, $X\in\mathfrak{X}(U)$, defined eventually on an open subset $U\subseteq \mathbb{R}^{n}$ (i.e., it admits a set of $n-1$ first integrals, $I_1,\dots,I_k, D_1,\dots, D_p\in\mathcal{C}^{\infty}(U,\mathbb{R})$, independent at least on an open subset $V \subseteq U$). Suppose that $\Gamma=\{\gamma(t)\subset V : 0\leq t\leq T \}$ is a $T-$periodic orbit of the system $\dot x=X(x)$. The idea for the stabilization procedure is to perturb the completely integrable system $\dot x=X(x)$, in such a way that the perturbed dynamical system becomes a dissipative dynamical system on $V$, which admits also $\Gamma$ as a periodic orbit, and moreover verifies the hypothesis of Theorem \eqref{SPO}. Note that using classical perturbation methods, the persistence of periodic orbits after perturbations, follows as a consequence of the implicit function theorem. The method introduced in this section, provide for the class of completely integrable dynamical system, an explicit perturbation which preserve (under reasonable conditions) an a-priori given periodic orbit.

\begin{theorem}\label{OST}
Let $\dot x= X(x)$ be a completely integrable dynamical system generated by a smooth vector field $X\in\mathfrak{X}(U)$ defined eventually on an open subset $U\subseteq \mathbb{R}^{n}$, and let $k,p\in\mathbb{N}$ be two natural numbers, with $k+p=n-1$, such that there exist $n-1$ first integrals $I_1,\dots,I_k, D_1,\dots, D_p\in\mathcal{C}^{\infty}(U,\mathbb{R})$, independent on an open subset $V \subseteq U$. Suppose the system $\dot x= X(x)$ admits a $T-$periodic orbit $\Gamma=\{\gamma(t)\subset V : 0\leq t\leq T \}$ such that:
\begin{itemize}
\item $\Gamma\subset ID^{-1}(\{0\})$, and $0 \in \mathbb{R}^{n-1}$ is a regular value of the map $$ID=(I_1,\dots,I_k,D_1,\dots,D_p):V\subseteq \mathbb{R}^{n}\rightarrow \mathbb{R}^{n-1},$$
\item $\nabla I_1(\gamma(t)),\dots,\nabla I_k(\gamma(t)),\nabla D_1(\gamma(t)),\dots, \nabla D_p (\gamma(t)), X(\gamma(t))$ are linearly\\
 independent for each $0\leq t \leq T$.
\end{itemize}

If moreover, $0 \in \mathbb{R}^{k}$ is a regular value of the map $I=(I_1,\dots,I_k):V \subseteq \mathbb{R}^{n}\rightarrow \mathbb{R}^{k},$ then for any choice of smooth functions $h_1,\dots, h_p \in\mathcal{C}^{\infty}(V,\mathbb{R})$ such that
$$\int_{0}^{T}h_1 (\gamma(s))ds<0, \dots, \int_{0}^{T}h_p (\gamma(s))ds<0,$$
$\Gamma$, as a periodic orbit of the dissipative dynamical system $\dot x =X(x)+X_{0}(x)$, $x\in V$,
$$
X_{0}=\left\| \bigwedge_{i=1}^{p} \nabla D_i\wedge\bigwedge_{j=1}^{k} \nabla I_j \right\|_{n-1}^{-2}\cdot\sum_{i=1}^{p}(-1)^{n-i}h_i D_i \Theta_i, 
$$
$$
\Theta_i = \star\left[ \bigwedge_{j=1, j\neq i}^{p} \nabla D_j \wedge \bigwedge_{l=1}^{k} \nabla I_l  \wedge\star\left(\bigwedge_{j=1}^{p} \nabla D_j\wedge\bigwedge_{l=1}^{k} \nabla I_l \right)\right],
$$
is orbitally phase asymptotically stable, with respect to perturbations along the invariant manifold $I^{-1}(\{0\})$.

On the other hand, for any choice of smooth functions
$k_1,\dots, k_p \in\mathcal{C}^{\infty}(V,\mathbb{R}),$ such that there exists $i_0 \in \{1,\dots,p\}$ for which $$\int_{0}^{T}k_{i_0} (\gamma(s))ds>0,$$
$\Gamma$, as a periodic orbit of the dissipative dynamical system $\dot x =X(x)+X_{0}(x)$, $x\in V$,
$$
X_{0}=\left\| \bigwedge_{i=1}^{p} \nabla D_i\wedge\bigwedge_{j=1}^{k} \nabla I_j \right\|_{n-1}^{-2}\cdot\sum_{i=1}^{p}(-1)^{n-i}k_i D_i \Theta_i, 
$$
$$
\Theta_i = \star\left[ \bigwedge_{j=1, j\neq i}^{p} \nabla D_j \wedge \bigwedge_{l=1}^{k} \nabla I_l  \wedge\star\left(\bigwedge_{j=1}^{p} \nabla D_j\wedge\bigwedge_{l=1}^{k} \nabla I_l \right)\right],
$$
is an unstable periodic orbit.

\end{theorem}
\begin{proof}
Let us show first that the perturbed system $\dot x = X(x)+ X_0 (x)$, is a codimension-one dissipative dynamical system. In order to do that, will be enough to prove that the vector field $X+X_0$ conserves $I_1,\dots,I_k$ and dissipates $D_1,\dots, D_p$. For a unified notation, let us denote for each $j\in\{1,\dots,p\}$, $u_j = h_j$, for the fist hypothesis of Theorem \eqref{OST}, and respectively $u_j = k_j$, for the second hypothesis of Theorem \eqref{OST}. 

Recall from Remark \eqref{ret} that:
\begin{equation*}
\left\{\begin{array}{l}
\mathcal{L}_{X_0}I_1= \dots = \mathcal{L}_{X_0}I_k = 0,\\
\mathcal{L}_{X_0}D_1 = u_1 D_1, \dots, \mathcal{L}_{X_0}D_p = u_p D_p.\\
\end{array}\right.
\end{equation*}

Hence, for each $i\in\{1,\dots,k\}$ and respectively $j\in\{1,\dots,p\}$, we obtain
\begin{equation*}
\left\{\begin{array}{l}
\mathcal{L}_{X+X_0}I_i= \mathcal{L}_{X}I_i + \mathcal{L}_{X_0}I_i=0+0=0,\\
\mathcal{L}_{X+X_0}D_j = \mathcal{L}_{X}D_j+\mathcal{L}_{X_0}D_j= 0+ u_j D_j = u_j D_j,\\
\end{array}\right.
\end{equation*}
and consequently, $\dot x = X(x)+X_{0}(x)$, $x\in V$, is a codimension-one dissipative dynamical system.

Recall that $\Gamma=\{\gamma(t)\subset V : 0\leq t\leq T \}$ is a periodic orbit of the dynamical system $\dot x = X(x)+X_{0}(x)$ too, since the hypothesis $\Gamma\subset ID^{-1}(\{0\})$ implies that $D_i \circ \gamma =0$, for every $i\in\{1,\dots, p\}$, and consequently $X_{0} (\gamma(t))=0$, for every $t\in[0,T]$.

Note that since $(X+X_0)(\gamma(t))=X(\gamma(t))$, for every $t\in[0,T]$, the condition that $\nabla I_1(\gamma(t)),\dots,\nabla I_k(\gamma(t)),\nabla D_1(\gamma(t)),\dots, \nabla D_p (\gamma(t)), X(\gamma(t))$ are linearly independent for each $0\leq t \leq T$, is obviously equivalent with the condition that $$\nabla I_1(\gamma(t)),\dots,\nabla I_k(\gamma(t)),\nabla D_1(\gamma(t)),\dots, \nabla D_p (\gamma(t)), (X+X_0)(\gamma(t))$$ are linearly independent for each $0\leq t \leq T$.

Now the conclusions of Theorem \eqref{OST} follow by applying the Theorem \eqref{SPO} for the codimension-one dissipative dynamical system $\dot x = X(x)+ X_0 (x)$, $x\in V$, and respectively the $T-$periodic orbit $\Gamma=\{\gamma(t)\subset V : 0\leq t\leq T \}$. 
\end{proof}

\begin{remark}\label{remimp}
In the hypothesis of the Theorem \eqref{OST}, note that:
\begin{itemize}
\item a consequence of the Remark \eqref{CIS} is that the set of points $x\in U$ such that $$\left\| \bigwedge_{i=1}^{p} \nabla D_i (x)\wedge\bigwedge_{j=1}^{k} \nabla I_j (x)\right\|_{n-1}=0,$$ is a subset of the equilibrium points of the completely integrable vector field $X$.

\item the condition $\Gamma\subset ID^{-1}(\{0\})$ implies that \textbf{for any choice} of smooth functions $h_1,\dots, h_p \in\mathcal{C}^{\infty}(V,\mathbb{R})$, the control vector field $X_0 \in \mathfrak{X}(V)$, given by 
$$
X_{0}=\left\| \bigwedge_{i=1}^{p} \nabla D_i\wedge\bigwedge_{j=1}^{k} \nabla I_j \right\|_{n-1}^{-2}\cdot\sum_{i=1}^{p}(-1)^{n-i}h_i D_i \Theta_i, 
$$
$$
\Theta_i = \star\left[ \bigwedge_{j=1, j\neq i}^{p} \nabla D_j \wedge \bigwedge_{l=1}^{k} \nabla I_l  \wedge\star\left(\bigwedge_{j=1}^{p} \nabla D_j\wedge\bigwedge_{l=1}^{k} \nabla I_l \right)\right],
$$
verifies that $X_{0} (\gamma(t))=0$, for every $t\in[0,T]$;

\item each of the smooth functions $h_1,\dots, h_p \in\mathcal{C}^{\infty}(V,\mathbb{R})$ might be chosen of the type e.g., $h(x)=-(\psi^2 (x)+ c)$, $x\in V$, with $\psi\in\mathcal{C}^{\infty}(V,\mathbb{R})$ and $c\in (0,\infty)$, since $$\int_{0}^{T}h (\gamma(s))ds=-\int_{0}^{T}\psi^2 (\gamma(s))ds -Tc \leq -Tc<0;$$

\item the smooth function $k_{i_0} \in\mathcal{C}^{\infty}(V,\mathbb{R})$ might be chosen of type e.g., $k(x)=\phi^2 (x)+ c$, $x\in V$, with $\phi\in\mathcal{C}^{\infty}(V,\mathbb{R})$ and $c\in (0,\infty)$, since $$\int_{0}^{T}k (\gamma(s))ds=\int_{0}^{T}\phi^2 (\gamma(s))ds +Tc \geq  Tc>0.$$
\end{itemize}
\end{remark}
\bigskip

Let us now illustrate the above stabilization result in the case of the harmonic oscillator. We consider this simple example in order to point out that different choices of $I's$ and $D's$ may generate different domains of definition for the perturbed vector field.

\begin{example}
Let us consider the family of harmonic oscillators, described by the three dimensional vector field $$X(x,y,z)=y \partial_{x}-x \partial_{y} \in\mathfrak{X}(\mathbb{R}^3).$$

The induced dynamical system, \begin{equation}\label{HO}\dot{\mathbf{x}}=X(\mathbf{x}), \ \mathbf{x}=(x,y,z)\in\mathbb{R}^{3},\end{equation} admits a $2\pi-$periodic orbit given by $\Gamma=\{\gamma(t)=(\sin t,\cos t,0): 0\leq t\leq 2\pi\}$.

Moreover, the system \eqref{HO} is completely integrable, since it has two independent first integrals, namely
$$
I_1 (x,y,z)=x^2 +y^2 -1, \ I_2 (x,y,z)=z.
$$

In order to apply the Theorem \eqref{OST}, the candidates for the functions $I$ and $D$ are the first integrals $I_1$ and respectively $I_2$. Consequently, we have two cases, namely $I=I_1$ and $D=I_2$, and respectively $I=I_2$ and $D=I_1$.

\medskip
$\diamond$ \textbf{Let us now analyze the first case, namely $I=I_1$ and $D=I_2$.}
\medskip

By straightforward computations we obtain that the vector field $X_0$ from Theorem \eqref{OST}, in this case has the expression $$X_0 (x,y,z)=z u(x,y,z)\partial_{z}, \ (x,y,z)\in\mathbb{R}^{3},$$ and consequently it verifies the condition $X_0 \circ\gamma=0$, for any smooth real function $u\in\mathcal{C}^{\infty}(\mathbb{R}^{3},\mathbb{R})$.

Consequently, the perturbed system 
\begin{equation}\label{PSTR}
\dot{\mathbf{x}}=X(\mathbf{x})+X_{0}(\mathbf{x}), \ \mathbf{x}=(x,y,z)\in\mathbb{R}^{3},
\end{equation}
is a codimension-one dissipative dynamical system associated to $I,D,u\in\mathcal{C}^{\infty}(\mathbb{R}^{3},\mathbb{R})$, i.e., $\mathcal{L}_{X+X_{0}}I=0$, and respectively $\mathcal{L}_{X+X_{0}}D=uD$.

Since $X_0 \circ\gamma=0$, for any smooth real function $u\in\mathcal{C}^{\infty}(\mathbb{R}^{3},\mathbb{R})$, we obtain that $\Gamma$ is a periodic orbit of the dissipative system $\dot{\mathbf{x}}=X(\mathbf{x})+X_{0}(\mathbf{x})$, for any smooth real function $u\in\mathcal{C}^{\infty}(\mathbb{R}^{3},\mathbb{R})$. 

Moreover, the rest of hypothesis of Theorem \eqref{OST} are similar to those of Theorem \eqref{SPO}, and were already been verified in Example \eqref{DHO}
for the vector field $$X(x,y,z)+X_0 (x,y,z)= y \partial_{x}-x \partial_{y}+ zu(x,y,z)\partial_{z}.$$

Hence, by Theorem \eqref{OST}, we obtain the following conclusions:
\begin{itemize}
\item for any smooth function $u\in \mathcal{C}^{\infty}(\mathbb{R}^{3},\mathbb{R})$ such that $\int_{0}^{2\pi}u(\sin t,\cos t,0)dt<0$, the periodic orbit $\Gamma$ of the associated perturbed system \eqref{PSTR} is orbitally phase asymptotically stable, with respect to perturbations along the cylinder $I^{-1}(\{0\})$;
\item for any smooth function $u\in \mathcal{C}^{\infty}(\mathbb{R}^{3},\mathbb{R})$ such that $\int_{0}^{2\pi}u(\sin t,\cos t,0)dt>0$, the periodic orbit $\Gamma$ of the associated perturbed system \eqref{PSTR} is unstable.
\end{itemize}

$\diamond$ $\diamond$ \textbf{Let us now analyze the second case, namely $I=I_2$ and $D=I_1$.} 
\medskip

By straightforward computations we obtain that the vector field $X_0$ from Theorem \eqref{OST}, in this case has the expression 
$$
X_0 (x,y,z)=\dfrac{u(x,y,z)(x^2 +y^2 -1)}{2(x^2 +y^2)}\left(x\partial_{x}+y\partial_{y}\right), \ (x,y,z)\in V:=\mathbb{R}^{3}\setminus\{(0,0,z): z\in\mathbb{R}\},
$$ 
and consequently it verifies the condition $X_0 \circ\gamma=0$, for any smooth real function $u\in\mathcal{C}^{\infty}(V,\mathbb{R})$.

Consequently, the perturbed system 
\begin{equation}\label{PSTR1}
\dot{\mathbf{x}}=X(\mathbf{x})+X_{0}(\mathbf{x}), \ \mathbf{x}=(x,y,z)\in V,
\end{equation}
is a codimension-one dissipative dynamical system associated to $I,D,u$, i.e., $\mathcal{L}_{X+X_{0}}I=0$, and respectively $\mathcal{L}_{X+X_{0}}D=uD$.

Since $X_0 \circ\gamma=0$, for any smooth real function $u\in\mathcal{C}^{\infty}(V,\mathbb{R})$, we obtain that $\Gamma$ is a periodic orbit of the dissipative system $\dot{\mathbf{x}}=X(\mathbf{x})+X_{0}(\mathbf{x})$, for any smooth real function $u\in\mathcal{C}^{\infty}(V,\mathbb{R})$. 

Moreover, the rest of hypothesis of Theorem \eqref{OST} are obviously verified, since they are similar with those from the previous case. Note that in this case, the perturbed vector field is given by 
\begin{align*}
X(x,y,z)+X_0 (x,y,z)&= \left[ \dfrac{u(x,y,z)x(x^2 +y^2 -1)}{2(x^2 +y^2)}+y\right]\partial_{x}\\
&+\left[ \dfrac{u(x,y,z)y(x^2 +y^2 -1)}{2(x^2 +y^2)}-x \right]\partial_{y}.
\end{align*}

Hence, by Theorem \eqref{OST}, we obtain the following conclusions:
\begin{itemize}
\item for any smooth function $u\in \mathcal{C}^{\infty}(V,\mathbb{R})$ such that $\int_{0}^{2\pi}u(\sin t,\cos t,0)dt<0$, the periodic orbit $\Gamma$ of the associated perturbed system \eqref{PSTR1} is orbitally phase asymptotically stable, with respect to perturbations along the plane $I^{-1}(\{0\})$;
\item for any smooth function $u\in \mathcal{C}^{\infty}(V,\mathbb{R})$ such that $\int_{0}^{2\pi}u(\sin t,\cos t,0)dt>0$, the periodic orbit $\Gamma$ of the associated perturbed system \eqref{PSTR1} is unstable.
\end{itemize}
\end{example}

Let us now illustrate the main result in the case of a mechanical dynamical system, namely Euler's equations from the free rigid body dynamics. 

\begin{example}
Let us recall that Euler's equations in terms of the rigid body angular momenta are generated by the vector field
$$
X(x,y,z)=\left( \dfrac{1}{I_3}-\dfrac{1}{I_2}\right)yz \partial_{x}+ \left( \dfrac{1}{I_1}-\dfrac{1}{I_3}\right)zx\partial_{y}+ 
\left( \dfrac{1}{I_2}-\dfrac{1}{I_1}\right)xy\partial_{z}\in\mathfrak{X}(\mathbb{R}^3),
$$
where $I_1, I_2, I_3$ are the moments of inertia in the principal axis frame of the rigid body. We will suppose in the following that $I_1 > I_2 > I_3 > 0$.

The induced dynamical system, \begin{equation}\label{HOO}\dot{\mathbf{x}}=X(\mathbf{x}), \ \mathbf{x}=(x,y,z)\in\mathbb{R}^{3},\end{equation} is completely integrable, since it has two independent first integrals, namely
$$
F_1 (x,y,z)=\dfrac{1}{2}\left(\dfrac{x^2}{I_1} +\dfrac{y^2}{I_2} +\dfrac{z^2}{I_3} \right), \ F_2 (x,y,z)=\dfrac{1}{2}\left(x^2 + y^2 + z^2 \right).
$$

Recall that there exists an open and dense subset of the image of the map $(F_1,F_2):\mathbb{R}^3 \rightarrow \mathbb{R}^2$, such that each fiber of any element $(h,c)$ from this set, corresponds to periodic orbits of Euler's equations. Moreover, any such element is a regular value of $(F_1,F_2)$, as well as its components for the corresponding maps, $F_1$ and respectively $F_2$. For more details see, e.g., \cite{dgt}.

Let $(h,c)\in \mathbb{R}^2$ belongs to the above mention set, and let us denote $$J_1 (x,y,z):=F_1(x,y,z)-h, \quad J_2 (x,y,z):= F_2 (x,y,z)- c.$$ 

In the above conditions, the set $$\Gamma=\left\{\gamma(t)=(\gamma_1 (t),\gamma_2 (t),\gamma_3 (t)): 0\leq t\leq \dfrac{4K\sqrt{I_1 I_2 I_3}}{\sqrt{2(I_2 - I_3)(h I_1 - c)}}\right\},$$ where 

\begin{align*}
\gamma_1 (t)&=\sqrt{\dfrac{2I_1 (c-h I_3)}{I_1 - I_3}} \cdot \operatorname{cn}\left(\sqrt{\dfrac{2(I_2 - I_3)(h I_1 - c)}{I_1 I_2 I_3}}\cdot t;k\right),\\
\gamma_2 (t)&=\sqrt{\dfrac{2I_2 (c-h I_3)}{I_2 - I_3}} \cdot \operatorname{sn}\left(\sqrt{\dfrac{2(I_2 - I_3)(h I_1 - c)}{I_1 I_2 I_3}}\cdot t;k\right),\\
\gamma_3 (t)&=-\sqrt{\dfrac{2I_3 (-c+h I_1)}{I_1 - I_3}} \cdot \operatorname{dn}\left(\sqrt{\dfrac{2(I_2 - I_3)(h I_1 - c)}{I_1 I_2 I_3}}\cdot t;k\right),\\
k&=\sqrt{\dfrac{(h I_3 - c)(I_1 - I_2)}{(h I_1 - c)(I_3 - I_2)}},\\
K&=\int_{0}^{1}\dfrac{dt}{\sqrt{(1-t^2)(1-k^2 t^2)}},
\end{align*}
is a $\dfrac{4K\sqrt{I_1 I_2 I_3}}{\sqrt{2(I_2 - I_3)(h I_1 - c)}}-$periodic orbit of the dynamical system \eqref{HOO}, which belongs to the common zero level set of the first integrals $J_1$ and $J_2$. See for details, e.g., \cite{dgt}. Recall that the above parameterization of $\Gamma$ is given in terms of Jacobi elliptic functions.

By straightforward computations we get that $\nabla J_1 (\mathbf{x}), \nabla J_2 (\mathbf{x}), X (\mathbf{x})$, are linearly dependent vectors if and only if $\mathbf{x}$ is an equilibrium point of the dynamical system \eqref{HOO}. Hence, the vectors $\nabla J_1 (\gamma(t)), \nabla J_2 (\gamma(t)), X (\gamma(t))$, are linearly independent for each $t\in \left[0, \dfrac{4K\sqrt{I_1 I_2 I_3}}{\sqrt{2(I_2 - I_3)(h I_1 - c)}}\right].$

In order to apply the Theorem \eqref{OST}, the candidates for the functions $I$ and $D$ are the first integrals $J_1$ and respectively $J_2$. Consequently, we have two cases, namely $I=J_1$ and $D=J_2$, and respectively $I=J_2$ and $D=J_1$.

\medskip
$\diamond$ \textbf{Let us now analyze the first case, namely $I=J_1$ and $D=J_2$.}
\medskip

By straightforward computations we obtain that the vector field $X_0$ from Theorem \eqref{OST}, in this case has the expression 

\begin{align*}
X_0 (x,y,z)&=\dfrac{u(x,y,z)\left[\dfrac{1}{2}(x^2 +y^2 +z^2)-c \right]}{\left[xy\left(\dfrac{1}{I_1}-\dfrac{1}{I_2}\right)\right]^2 +\left[yz\left(\dfrac{1}{I_2}-\dfrac{1}{I_3}\right)\right]^2+\left[xz\left(\dfrac{1}{I_1}-\dfrac{1}{I_3}\right)\right]^2}\\
&\cdot \{ x\left[ \dfrac{1}{I_2}\left(\dfrac{1}{I_2}-\dfrac{1}{I_1}\right) y^2 + \dfrac{1}{I_3}\left(\dfrac{1}{I_3}-\dfrac{1}{I_1}\right) z^2\right]\partial_{x} \\
&+ y\left[ \dfrac{1}{I_1}\left(\dfrac{1}{I_1}-\dfrac{1}{I_2}\right) x^2 + \dfrac{1}{I_3}\left(\dfrac{1}{I_3}-\dfrac{1}{I_2}\right) z^2\right]\partial_{y}\\
&+ z\left[ \dfrac{1}{I_1}\left(\dfrac{1}{I_1}-\dfrac{1}{I_3}\right) x^2 + \dfrac{1}{I_2}\left(\dfrac{1}{I_2}-\dfrac{1}{I_3}\right) y^2\right]\partial_{z} \}
, \ (x,y,z)\in V,
\end{align*}
where $V:=\mathbb{R}^3 \setminus\{\{(x,0,0): x\in\mathbb{R}\}\cup \{(0,y,0): y\in\mathbb{R}\} \cup \{(0,0,z): z\in\mathbb{R}\}\}$ (note that $V$ is exactly the complement of the set of equilibrium states of Euler's equations). 

Recall that the vector field $X_0$ verifies the condition $X_0 \circ\gamma=0$, for any smooth real function $u\in\mathcal{C}^{\infty}(V,\mathbb{R})$, since $\Gamma$ belongs to the zero level set of the first integral $J_2$.

Consequently, the perturbed system 
\begin{equation}\label{PSTRA}
\dot{\mathbf{x}}=X(\mathbf{x})+X_{0}(\mathbf{x}), \ \mathbf{x}=(x,y,z)\in\mathbb{R}^{3},
\end{equation}
is a codimension-one dissipative dynamical system associated to $I,D,u$, i.e., $\mathcal{L}_{X+X_{0}}I=0$, and respectively $\mathcal{L}_{X+X_{0}}D=uD$.

Since $X_0 \circ\gamma=0$, for any smooth real function $u\in\mathcal{C}^{\infty}(V,\mathbb{R})$, we obtain that $\Gamma$ is a periodic orbit of the dissipative system $\dot{\mathbf{x}}=X(\mathbf{x})+X_{0}(\mathbf{x})$, for any smooth real function $u\in\mathcal{C}^{\infty}(V,\mathbb{R})$. 

Since all the hypothesis of Theorem \eqref{OST} are verified, we obtain the following conclusions. If one denotes $T:=\dfrac{4K\sqrt{I_1 I_2 I_3}}{\sqrt{2(I_2 - I_3)(h I_1 - c)}}$, then
\begin{itemize}
\item for any smooth function $u\in \mathcal{C}^{\infty}(V,\mathbb{R})$ such that $\int_{0}^{T}u(\gamma_1 (t),\gamma_2 (t),\gamma_3 (t))dt<0$, the periodic orbit $\Gamma$ of the associated perturbed system \eqref{PSTRA} is orbitally phase asymptotically stable, with respect to perturbations along the ellipsoid 
$$I^{-1}(\{0\})=\left\{(x,y,z): \ \dfrac{1}{2}\left(\dfrac{x^2}{I_1} +\dfrac{y^2}{I_2} +\dfrac{z^2}{I_3} \right)-h=0 \right\};$$
\item for any smooth function $u\in \mathcal{C}^{\infty}(V,\mathbb{R})$ such that $\int_{0}^{T}u(\gamma_1 (t),\gamma_2 (t),\gamma_3 (t))dt>0$, the periodic orbit $\Gamma$ of the associated perturbed system \eqref{PSTRA} is unstable.
\end{itemize}

$\diamond$ $\diamond$ \textbf{Let us now analyze the second case, namely $I=J_2$ and $D=J_1$.} 
\medskip

By straightforward computations we obtain that the vector field $X_0$ from Theorem \eqref{OST}, in this case has the expression 
\begin{align*}
X_0 (x,y,z)&=\dfrac{u(x,y,z)\left[\dfrac{1}{2}\left(\dfrac{x^2}{I_1} +\dfrac{y^2}{I_2} +\dfrac{z^2}{I_3}\right)-h \right]}{\left[xy\left(\dfrac{1}{I_1}-\dfrac{1}{I_2}\right)\right]^2 +\left[yz\left(\dfrac{1}{I_2}-\dfrac{1}{I_3}\right)\right]^2+\left[xz\left(\dfrac{1}{I_1}-\dfrac{1}{I_3}\right)\right]^2}\\
&\cdot \{ x\left[ \left(\dfrac{1}{I_1}-\dfrac{1}{I_2}\right) y^2 + \left(\dfrac{1}{I_1}-\dfrac{1}{I_3}\right) z^2\right]\partial_{x} \\
&+ y\left[ \left(\dfrac{1}{I_2}-\dfrac{1}{I_1}\right) x^2 + \left(\dfrac{1}{I_2}-\dfrac{1}{I_3}\right) z^2\right]\partial_{y}\\
&+ z\left[ \left(\dfrac{1}{I_3}-\dfrac{1}{I_1}\right) x^2 + \left(\dfrac{1}{I_3}-\dfrac{1}{I_2}\right) y^2\right]\partial_{z} \}
, \ (x,y,z)\in V,
\end{align*}
where $V:=\mathbb{R}^3 \setminus\{\{(x,0,0): x\in\mathbb{R}\}\cup \{(0,y,0): y\in\mathbb{R}\} \cup \{(0,0,z): z\in\mathbb{R}\}\}$ (note that $V$ is exactly the complement of the set of equilibrium states of Euler's equations). 

Recall that the vector field $X_0$ verifies the condition $X_0 \circ\gamma=0$, for any smooth real function $u\in\mathcal{C}^{\infty}(V,\mathbb{R})$, since $\Gamma$ belongs to the zero level set of the first integral $J_1$.

Consequently, the perturbed system 
\begin{equation}\label{PSTRB}
\dot{\mathbf{x}}=X(\mathbf{x})+X_{0}(\mathbf{x}), \ \mathbf{x}=(x,y,z)\in\mathbb{R}^{3},
\end{equation}
is a codimension-one dissipative dynamical system associated to $I,D,u$, i.e., $\mathcal{L}_{X+X_{0}}I=0$, and respectively $\mathcal{L}_{X+X_{0}}D=uD$.

Since $X_0 \circ\gamma=0$, for any smooth real function $u\in\mathcal{C}^{\infty}(V,\mathbb{R})$, we obtain that $\Gamma$ is a periodic orbit of the dissipative system $\dot{\mathbf{x}}=X(\mathbf{x})+X_{0}(\mathbf{x})$, for any smooth real function $u\in\mathcal{C}^{\infty}(V,\mathbb{R})$. 

Since all the hypothesis of Theorem \eqref{OST} are verified, we obtain the following conclusions. If one denotes $T:=\dfrac{4K\sqrt{I_1 I_2 I_3}}{\sqrt{2(I_2 - I_3)(h I_1 - c)}}$, then
\begin{itemize}
\item for any smooth function $u\in \mathcal{C}^{\infty}(V,\mathbb{R})$ such that $\int_{0}^{T}u(\gamma_1 (t),\gamma_2 (t),\gamma_3 (t))dt<0$, the periodic orbit $\Gamma$ of the associated perturbed system \eqref{PSTRB} is orbitally phase asymptotically stable, with respect to perturbations along the sphere 
$$I^{-1}(\{0\})=\left\{(x,y,z): \ \dfrac{1}{2}\left(x^2 + y^2 + z^2 \right)-c =0 \right\};$$
\item for any smooth function $u\in \mathcal{C}^{\infty}(V,\mathbb{R})$ such that $\int_{0}^{T}u(\gamma_1 (t),\gamma_2 (t),\gamma_3 (t))dt>0$, the periodic orbit $\Gamma$ of the associated perturbed system \eqref{PSTRB} is unstable.
\end{itemize}

\end{example}

\subsection*{Acknowledgment}
This work was supported by a grant of the Romanian National Authority for Scientific Research, CNCS-UEFISCDI, project number PN-II-RU-TE-2011-3-0103.

\bigskip
\bigskip

\noindent {\sc R.M. Tudoran}\\
West University of Timi\c soara\\
Faculty of Mathematics and Computer Science\\
Department of Mathematics\\
Blvd. Vasile P\^arvan, No. 4\\
300223 - Timi\c soara, Rom\^ania.\\
E-mail: {\sf tudoran@math.uvt.ro}\\
\medskip


\begin{thebibliography}{99}

\bibitem{rtud} {\footnotesize \textsc{P. Birtea, M. Bolean\c tu, M. Puta and R.M. Tudoran}, Asymptotic stability for a class of metriplectic systems, \textit{J. Math. Phys.} 48(8)(2007), 082703. }

\bibitem{bc} {\footnotesize \textsc{P. Birtea and D. Com\u anescu}, Asymptotic stability of dissipated Hamilton-Poisson systems, \textit{SIAM J. Appl. Dyn. Syst.} 8(3)(2009), 967--976. }

\bibitem{dgt} {\footnotesize \textsc{C. D\u anias\u a, A. G\^irban and R.M. Tudoran}, New aspects on the geometry and dynamics of quadratic Hamiltonian systems on $(\mathfrak{so}(3))^{*}$, \textit{Int. J. Geom. Methods Mod. Phys.}, 8(8)(2011), 1695--1721. }

\bibitem{gasul} {\footnotesize \textsc{A. Gasull, H. Giacomini and M. Grau}, On the stability of periodic orbits for differential systems in $\mathbb{R}^{n}$, \textit{Discrete and Continuous Dynamical Systems - Series B}, 10(2/3)(2008), 495--509. }

\bibitem{hartman} {\footnotesize \textsc{P. Hartman}, \textit{Ordinary Differential Equations}, Classics in Applied Mathematics, vol. 38, SIAM, 2002. }

\bibitem{moser} {\footnotesize \textsc{J. Moser and E. Zehnder}, \textit{Notes on Dynamical Systems}, Courant Lecture Notes in Mathematics, vol. 12, American Mathematical Society, 2005. }

\bibitem{ratiurazvan} {\footnotesize \textsc{T.S. Ratiu, R.M. Tudoran, L. Sbano, E. Sousa Dias and G. Terra}, \textit{Geometric Mechanics and Symmetry: the Peyresq Lectures; Chapter II: A Crash Course in Geometric Mechanics}, pp. 23--156, London Mathematical Society Lecture Notes Series, vol. 306, Cambridge University Press 2005. }

\bibitem{rtudoran} {\footnotesize \textsc{R.M. Tudoran}, On local generators of affine distributions on Riemannian manifolds, \textit{arXiv}: 1310.6512, (2013), 1--21. }

\bibitem{verhulst} {\footnotesize \textsc{F. Verhulst}, \textit{Nonlinear Differential Equations and Dynamical Systems}, second edition, Springer, 2006. }

\end{thebibliography}
\end{document}